\documentclass[3p, times]{elsarticle}

\usepackage{amssymb}

\usepackage{amsthm}

\usepackage{amsmath}

\usepackage{caption}

\usepackage{graphicx}

\usepackage{float} 

\usepackage{subfigure}

\usepackage{subcaption}

\usepackage[T1]{fontenc}

\biboptions{}

\newtheorem{theorem}{Theorem}[section]
\newtheorem{remark}[theorem]{Remark}
\newtheorem{example}[theorem]{Example}
\newtheorem{corollary}[theorem]{Corollary}

\newtheorem{lemma}[theorem]{Lemma}
\newtheorem{definition}[theorem]{Definition}

\newproof{pf}{Proof}

\begin{document}
\sloppy{}
	
\begin{frontmatter}
\title{Asymptotic stability and exponential stability for a class of impulsive neutral differential equations with discrete and distributed delays} 
\author{Jinyuan Pan}
\author{Guiling Chen\corref{cor1}}
\ead{guiling@swjtu.edu.cn}
\cortext[cor1]{Corresponding author}
\address{School of Mathematics, Southwest Jiaotong University, Chengdu 610031, PR China}	
\begin{abstract}
In this paper, we present sufficient conditions for asymptotic stability and exponential stability of a class of impulsive neutral differential equations with discrete and distributed delays. Our approaches are based on the method using fixed point theory, which do not resort to any Lyapunov functions or Lyapunov functionals. Our conditions do not require the differentiability of delays, nor do they ask for a fixed sign on the coefficient functions. Our results improve some previous ones in the literature. Examples are given to illustrate our main results. 
\end{abstract}

	
\end{frontmatter}

\section{Introduction}

Many real world problems in science and engineering can be modelled by neutral delay differential equations, such as delayed cellular neural network models\cite{ZYT-1, ZYT-2} and heat conduction in materials with decay memory\cite{AL}. 
The existence, uniqueness and stability problems of the neutral delay differential equations have been investigated by many authors, for example, Afonso et al.\cite{ABS}, Jin and Luo\cite{JL}, Mesmouli\cite{MAD} and Raffoul\cite{R-1, R-2}, etc. 

Lyapunov's direct method has been very effective in establishing stability results for a wide variety of differential equations. The success of Lyapunov's direct method depends on finding a suitable Lyapunov function or Lyapunov functional. However, it may be difficult to look for a good Lyapunov functional for some classes of delay differential equations. Therefore, an alternative may be explored to overcome some difficulties. It was proposed by Burton\cite{TAB} and his co-workers to use fixed point methods to study the stability problems for deterministic systems. Afterwards, a great number of classes of delay differential equations are studied by this method, see, for example, \cite{BZ, AD, GOS1, GOS2, GOS3}. It turned out that the fixed point method is a powerful technique in dealing with stability problems of deterministic delay differential equations. Furthermore, this approach possesses the advantage that it can yield the existence, uniqueness and stability criteria of the considered system in one step.

In addition to delay effects, impulsive effects are also likely to exist in some systems, which could stabilize or destabilize the systems. Therefore, it is necessary to take delay effects and impulsive effects into account on dynamical systems. Recently, many research have studied the stability of impulsive delay differential equations and obtained interesting results, for example, Mesmouli\cite{MBM}, Yan and zhao\cite{YJ}, Liu and Ramirez\cite{LXZ}, etc. 

To the best of author's knowledge, the fixed point method is mainly used to deal with the stability for scalar deterministic differential equations. However, there is not much work discussing stability behaviors of n-dimensional neutral delayed systems with variable coefficients. In this paper, we address asymptotic stability and exponential stability of a class of n-dimensional impulsive neutral differential equations with variable coefficients. 

This paper is organized as follows. The model is described and some basic preliminaries are presented in Section 2. Asymptotic stability of the system is studied in Section 3. Exponential stability of the system is investigated in Section 4. Examples are given to illustrate our main results in Section 5. 

\section{Model description and preliminaries}
Let $\mathbb{R}^{n}$ denote the n-dimensional Euclidean space and let $\left\|\cdot\right\|$
represent the 1-norm defined as the sum of the absolute values of its elements. $\mathbb{R}^{+}=[0, +\infty)$. 
$C(X, Y)$ corresponds to the space of continuous mappings
from the topological space $X$ to the topological space $Y$. 

We consider a class of nonlinear impulsive neutral delayed system with discrete and distributed delays
\begin{eqnarray}\label{main equation}
&&\left\{
\begin{array}{ll}
d\left[x_{i}(t)-\sum_{j=1}^{n}q_{ij}(t)x_{j}(t-\tau(t))\right]\\
\quad =\left[\sum_{j=1}^{n}c_{ij}(t)x_{j}(t)+\sum_{j=1}^{n}a_{ij}(t)f_{j}(x_{j}(t))+\sum_{j=1}^{n}b_{ij}(t)g_{j}(x_{j}(t-\delta(t)))
+\sum_{j=1}^{n}w_{ij}(t)\int_{t-r(t)}^{t}h_j(x_{j}(s))\, ds\right]\, dt, \\

\qquad\qquad\qquad\qquad\qquad\qquad\qquad\qquad\qquad\qquad\qquad\qquad\qquad\qquad\qquad\qquad\quad t\geq 0, \quad t\neq t_{k}, \\&\\

\Delta x_i(t_{k})=x_i(t_{k})-x_i(t^{-}_{k}), \quad t=t_{k}, \quad k=1, 2, 3, \cdots, \quad i=1, 2, 3, \cdots, n. 
\end{array}
\right. 
\end{eqnarray}
This can be written in a vector-matrix form as follows
\begin{eqnarray}\label{main equation1}
&&\left\{
\begin{array}{ll}
d\left[x(t)-Q(t)x(t-\tau(t))\right]=\left[C(t)x(t)+A(t)f(x(t))+B(t)g(x(t-\delta(t)))+W(t)\int_{t-r(t)}^{t}h(x(s))\, ds\right]\, dt, \\
\qquad\qquad\qquad\qquad\qquad\qquad\qquad\qquad\qquad\qquad\qquad\qquad\qquad\qquad\qquad\qquad t\geq 0, \quad t\neq t_{k}, \\&\\
\Delta x(t_{k})=x(t_{k})-x(t^{-}_{k}), \quad t=t_{k}, \quad k=1, 2, 3, \cdots, 
\end{array}
\right. 
\end{eqnarray}
where $ {x}(t)=(x_{1}(t), x_{2}(t), \cdots, x_{n}(t))^{T}\in C(\mathbb{R}^{+}, {\mathbb{R}}^{n}) $, $ Q(t)=(q_{ij}(t))_{n\times n} $, $ C(t)=(c_{ij}(t))_{n\times n} $, $ A(t)=(a_{ij}(t))_{n\times n} $, $ B(t)=(b_{ij}(t))_{n\times n} $, $ W(t)=(w_{ij}(t))_{n\times n} $, and $a_{ij}(t)$, $ b_{ij}(t)$, $c_{ij}(t)$, $w_{ij}(t)$, $q_{ij}(t)\in C(\mathbb{R}^{+}, \mathbb{R})$, $f(x(t))=(f_{1}(x_{1}(t)),$ $ f_{2}(x_{2}(t)), \cdots, f_{n}(x_{n}(t)))^{T}\in C(\mathbb{R}^{+}, {\mathbb{R}}^{n}) $, $ g(x(t))=(g_{1}(x_1(t)),  g_{2}(x_2(t)), \cdots, g_{n}(x_n(t)))^{T}\in C(\mathbb{R}^{+}, {\mathbb{R}}^{n}) $, $ h(x(t))=(h_{1}(x_1(t)),\,  h_{2}(x_2(t)), \cdots,  h_{n}(x_n(t)))^{T}\in C(\mathbb{R}^{+}, {\mathbb{R}}^{n}) $. $ \Delta x_{i}(t_{k})=x_{i}(t_{k}^+)-x_{i}(t_{k}^-)$ is the impulse at moment $ t_{k} $, and $ t_{1}<t_{2}<\cdots $ is a strictly increasing sequence such that $ \lim_{k\rightarrow \infty}t_{k}=+\infty$. $ x_{i}(t_{k}^+) $ and $ x_{i}(t_{k}^-) $ stand for the right-hand and left-hand limit of $ x_{i}(t) $ at $ t=t_{k} $. $\tau(t)$, $ \delta(t)$ and $ r(t)$ are nonegative continuous functions. Denote that $ \vartheta=\inf_{t\geq 0}\{ t-\tau(t), t-\delta(t), t-r(t)\} $. 

The initial condition for the system (\ref{main equation1}) is given by
\begin{eqnarray}\label{initial}
	&x(t)=\varphi(t), \quad \quad t\in[\vartheta, 0], &
\end{eqnarray}
where $ \varphi(t)=(\varphi_{1}(t), \varphi_{2}(t), \cdots, \varphi_{n}(t))^{T} \in C([\vartheta, 0], \mathbb{R}^{n}) $ is a continuous function with the norm defined by 
$$ \left\| \varphi\right\|=\sum_{i=1}^{n}\sup_{\vartheta\leq t\leq 0}|\varphi_{i}(t)|. $$

The solution $ x(t)\coloneqq x(t, 0, \varphi) $ of the system (\ref{main equation}) is, for the time $ t $, a piecewise continuous vector-valued function with the first kind discontinuity at the points $ t_{k} $ ($ k=1, 2, \cdots $), where it is right continuous, i.e. , 
\begin{eqnarray*}
x_{i}(t_{k}^{+})=x_{i}(t_{k}), \qquad x_{i}(t_{k})=x_{i}(t_{k}^{-})+\Delta x_{i}(t_{k}), \qquad i=1, 2, \cdots n, \qquad k=1, 2, \cdots. 
\end{eqnarray*}

\begin{definition}
		The trivial solution $x=0$ of the system {\rm{(\ref{main equation})}} is said to be stable, if, for any $\varepsilon>0$, there exists $\delta >0$ such that for any initial condition $\varphi\in C([\vartheta, 0], \mathbb{R}^{n})$ satisfying $\left\|\varphi\right\| < \delta$, $\left\|x(t, 0, \varphi) \right\|<\varepsilon $, $t\geq0$. 
\end{definition}
\begin{definition}
		The trivial solution $x=0$ of the system {\rm{(\ref{main equation})}} is said to be asymptotically stable if the trivial solution $x=0$ is stable, and for any initial condition $\varphi\in C([\vartheta, 0], \mathbb{R}^{n})$, $\lim_{t\rightarrow\infty}\|x(t, 0, \varphi)\|=0$ holds. 
\end{definition}
\begin{definition}
		The trivial solution $x=0$ of the system {\rm{(\ref{main equation})}} is said to be globally exponentially stable if there exists a pair of constants $\lambda>0$ and $C>0$ such that $\|x(t, 0, \varphi)\|\leq Ce^{-\lambda t}\|\varphi\|$ for $t\geq0$, where $\varphi\in C([\vartheta, 0], \mathbb{R}^{n})$. 
\end{definition}
\begin{theorem}\rm{\textbf{(Banach fixed point theorem)}}
		Let $(\mathcal{S}, \rho)$ be a complete metric space and let $P:\mathcal{S}\to \mathcal{S}$. If there is a constant $\alpha<1$ such that for each pair $\phi_1, \phi_2\in \mathcal{S}$ we have 
		\begin{eqnarray*}
			\rho(P\phi_1-P\phi_2)\leq \alpha\rho(\phi_1-\phi_2), 
		\end{eqnarray*}
		then there is one and only one point $\phi \in \mathcal{S}$ with $P\phi=\phi$. 
	\end{theorem}
\begin{lemma}
		$x(t)$ is a solution of the equation{\rm{ (\ref{main equation}) }}if and only if 
		\begin{eqnarray}\label{inteqn}
			x_i(t)&=&\sum_{j=1}^{n}q_{ij}(t)x_{j}(t-\tau(t))+\left[\varphi_i(0)-\sum_{j=1}^{n}q_{ij}(0)x_{j}(-\tau(0))\right]e^{-\int_{0}^{t}v_i(s)ds}+\sum_{0\leq t_k\leq t}I_{ik}(t_{k}, (Fx)_i(t_{k}))e^{-\int_{t_k}^{t}v_i(s)ds}\nonumber\\
			&&+\int_{0}^{t}e^{-\int_{s}^{t}v_i(u)du}\left[\sum_{j=1}^{n}\overline{c}_{ij}(s)x_{j}(s)+\sum_{j=1}^{n}a_{ij}(s)f_{j}(x_{j}(s))+\sum_{j=1}^{n}b_{ij}(s)g_{j}(x_{j}(s-\delta(s)))\right. \\
			&&\left. +\sum_{j=1}^{n}w_{ij}(s)\int_{s-r(s)}^{s}h_{j}(x_{j}(u))\, du-\sum_{j=1}^{n}v_i(s)q_{ij}(s)x_j(s-\delta(s))\right]ds\nonumber, 
		\end{eqnarray}
		where $(Fx)_i(t)=x_{i}(t)-\sum_{j=1}^{n}q_{ij}(t)x_{j}(t-\tau(t))$, $I_{ik}(t_{k}, (Fx)_i(t_{k}))= (Fx)_i(t_{k})- (Fx)_i(t^{-}_{k})$, $\overline{c}_{ij}(t)={c}_{ij}(t)$ $(i\neq j)$, $\overline{c}_{ii}(t)={c}_{ii}(t)+v_i(t)$. 
	\end{lemma}
	\begin{proof}
		We rewrite equation (\ref{main equation}) as
		\begin{eqnarray}\label{equvilent}
			d\left[x_{i}(t)-\sum_{j=1}^{n}q_{ij}(t)x_{j}(t-\tau(t))\right]&=&\Bigg[-v_i(t)x_{i}(t)+\sum_{j=1}^{n} \overline{c}_{ij}(t)x_{j}(t)+\sum_{j=1}^{n}a_{ij}(t)f_{j}(x_{j}(t))\nonumber\\
			 &&+\sum_{j=1}^{n}b_{ij}(t)g_{j}(x_{j}(t-\delta(t)))+\sum_{j=1}^{n}w_{ij}(t)\int_{t-r(t)}^{t}h_{j}(x_{j}(s))\, ds\Bigg]\, dt, \quad t\neq t_{k}, 
		\end{eqnarray}
		where $\overline{c}_{ij}(t)={c}_{ij}(t)$ $(i\neq j)$, $\overline{c}_{ii}(t)={c}_{ii}(t)+v_i(t)$, $v_i(t)$ is an auxiliary function we have chosen. \\
		
		Multiply both sides of (\ref{equvilent}) by $e^{\int_{0}^{t}v_i(s)\, ds}$, and integrate from $t_{k-1}$ to $t\in\left[t_{k-1}, t_k\right)\left(k=1, 2, 3, \cdots \right)$, we obtain
		\begin{eqnarray*}
			&&\left(x_{i}(t)-\sum_{j=1}^{n}q_{ij}(t)x_{j}(t-\tau(t))\right)e^{\int_{0}^{t}v_i(s)\, ds}-\left(x_{i}(t_{k-1})-\sum_{j=1}^{n}q_{ij}(t_{k-1})x_{j}(t_{k-1}-\tau(t_{k-1}))\right)e^{\int_{0}^{t_{k-1}}v_i(s)\, ds}\\
			&&\quad=\int_{t_{k-1}}^{t}\left[\sum_{j=1}^{n} \overline{c}_{ij}(s)x_{j}(s)+\sum_{j=1}^{n}a_{ij}(s)f_{j}(x_{j}(s))+\sum_{j=1}^{n}b_{ij}(s)g_{j}(x_{j}(s-\delta(s)))\right. \\
			&&\quad\quad+\left. \sum_{j=1}^{n}w_{ij}(s)\int_{s-r(s)}^{s}h_{j}(x_{j}(u))\, du-\sum_{j=1}^{n}v_i(s)q_{ij}(s)x_{j}(s-\tau(s))\right]e^{\int_{0}^{s}v_i(u)\, du}\, ds. 
		\end{eqnarray*}
		
		Thus, for $t\in\left[t_{k-1}, t_k\right)$, by putting $Fx:C(\mathbb{R}^{+}, {\mathbb{R}}^{n})\rightarrow C(\mathbb{R}^{+}, {\mathbb{R}}^{n})$, $ (Fx)_i(t)=x_{i}(t)-\sum_{j=1}^{n}q_{ij}(t)x_{j}(t-\tau(t))$, we obtain the following
		\begin{eqnarray}\label{B t}
		&&(Fx)_i(t)e^{\int_{0}^{t}v_i(s)\, ds}-(Fx)_i(t_{k-1})e^{\int_{0}^{t_{k-1}}v_i(s)\, ds}\nonumber\\
		&&\quad=(Fx)_i(t)e^{\int_{0}^{t}v_i(s)\, ds}-\left[I_{i(k-1)}(t_{k-1}, (Fx)_i(t_{k-1}))+(Fx)_i(t_{k-1}^-)\right]e^{\int_{0}^{t_{k-1}}v_i(s)\, ds}\nonumber\\
		&&\qquad+\int_{t_{k-1}}^{t}\left[\sum_{j=1}^{n} \overline{c}_{ij}(s)x_{j}(s)+\sum_{j=1}^{n}a_{ij}(s)f_{j}(x_{j}(s))+\sum_{j=1}^{n}b_{ij}(s)g_{j}(x_{j}(s-\delta(s)))\right. \\
			 &&\qquad+\left. \sum_{j=1}^{n}w_{ij}(s)\int_{s-r(s)}^{s}h_{j}(x_{j}(u))\, du-\sum_{j=1}^{n}v_i(s)q_{ij}(s)x_{j}(s-\tau(s))\right]e^{\int_{0}^{s}v_i(u)\, du}\, ds\nonumber, 
		\end{eqnarray}
		where $I_{ik}(t_{k}, (Fx)_i(t_{k})):C(\mathbb{R}^{+}, {\mathbb{R}}^{n})\rightarrow \mathbb{R}$, $I_{ik}(t_{k}, (Fx)_i(t_{k}))= (Fx)_i(t_{k})- (Fx)_i(t^{-}_{k})$. 
		
		Set $t=t_k-\varepsilon\left(\varepsilon>0\right)$ in (\ref{B t}), 
		\begin{eqnarray*}
			(Fx)_i(t_k-\varepsilon)e^{\int_{0}^{t_k-\varepsilon}v_i(s)\, ds}-(Fx)_i(t_{k-1})e^{\int_{0}^{t_{k-1}}v_i(t)}
			&=&\int_{t_{k-1}}^{t_k-\varepsilon}\left[\sum_{j=1}^{n} \overline{c}_{ij}(s)x_{j}(s)+\sum_{j=1}^{n}a_{ij}(s)f_{j}(x_{j}(s))\right. \\
			&&+\sum_{j=1}^{n}b_{ij}(s)g_{j}(x_{j}(s-\delta(s)))
			+\sum_{j=1}^{n}w_{ij}(s)\int_{s-r(s)}^{s}h_{j}(x_{j}(u))\, du\\
			&&\left. -\sum_{j=1}^{n}v_i(s)q_{ij}(s)x_{j}(s-\tau(s))\right]e^{\int_{0}^{s}v_i(u)\, du}\, ds, 
		\end{eqnarray*}
		and let $\varepsilon\rightarrow0$, we obtain
		\begin{eqnarray*}
			 (Fx)_i(t_k^{-})e^{\int_{0}^{t_k}v_i(s)\, ds}-(Fx)_i(t_{k-1})e^{\int_{0}^{t_{k-1}}v_i(s)\, ds}
			&=&\int_{t_{k-1}}^{t_k}\left[\sum_{j=1}^{n} \overline{c}_{ij}(s)x_{j}(s)+\sum_{j=1}^{n}a_{ij}(s)f_{j}(x_{j}(s))\right. \\
			&&+\sum_{j=1}^{n}b_{ij}(s)g_{j}(x_{j}(s-\delta(s)))
			+\sum_{j=1}^{n}w_{ij}(s)\int_{s-r(s)}^{s}h_{j}(x_{j}(u))\, du\\
			&&\left. -\sum_{j=1}^{n}v_i(s)q_{ij}(s)x_{j}(s-\tau(s))\right]e^{\int_{0}^{s}v_i(u)\, du}\, ds. 
		\end{eqnarray*}
		
		Backstep in this way, we get
		\begin{eqnarray*}
			(Fx)_i(t_{k-1}^{-})e^{\int_{0}^{t_{k-1}}v_i(s)\, ds}-(Fx)_i(t_{k-2})e^{\int_{0}^{t_{k-2}}v_i(s)\, ds}
			&=&\int_{t_{k-2}}^{t_{k-1}}\left[\sum_{j=1}^{n} \overline{c}_{ij}(s)x_{j}(s)+\sum_{j=1}^{n}a_{ij}(s)f_{j}(x_{j}(s))\right. \\
			&&+\sum_{j=1}^{n}b_{ij}(s)g_{j}(x_{j}(s-\delta(s)))
			+\sum_{j=1}^{n}w_{ij}(s)\int_{s-r(s)}^{s}h_{j}(x_{j}(u))\, du\\
			&&\left. -\sum_{j=1}^{n}v_i(s)q_{ij}(s)x_{j}(s-\tau(s))\right]e^{\int_{0}^{s}v_i(u)\, du}\, ds, \\
			&\cdots&\\
			&\cdots&\\
			&\cdots&\\
		\end{eqnarray*}
		\begin{eqnarray*}
		(Fx)_i(t_{1}^{-})e^{\int_{0}^{t_{1}}v_i(s)\, ds}-(Fx)_i(0)
			&=&\int_{0}^{t_1}\left[\sum_{j=1}^{n} \overline{c}_{ij}(s)x_{j}(s)+\sum_{j=1}^{n}a_{ij}(s)f_{j}(x_{j}(s))\right. \\
			&&+\sum_{j=1}^{n}b_{ij}(s)g_{j}(x_{j}(s-\delta(s)))
			+\sum_{j=1}^{n}w_{ij}(s)\int_{s-r(s)}^{s}h_{j}(x_{j}(u))\, du\\
			&&\left. -\sum_{j=1}^{n}v_i(s)q_{ij}(s)x_{j}(s-\tau(s))\right]e^{\int_{0}^{s}v_i(u)\, du}\, ds. 
		\end{eqnarray*}
		By recursive substitution into (\ref{B t}), the solution $x(t)$ must satisfy (\ref{inteqn}) . 
	\end{proof}
	
To obtain our results, we suppose the following conditions are satisfied:
\begin{enumerate}
	\item[\rm{(A1)}] the delays $ \tau(t)$, $\delta(t)$ and $r(t)$ are continuous functions such that $ t-\tau(t)\rightarrow \infty $, $ t-\delta(t)\rightarrow \infty $ and $ t-r(t)\rightarrow \infty $ as $ t\rightarrow \infty $. 
	\item[{\rm{(A2)}}] For $ j=1, 2, 3, \cdots, n $, the mappings $f_{j}(\cdot) $, $g_{j}(\cdot) $, and $h_{j}(\cdot) $ satisfy $ f_{j}(0)\equiv 0 $, $ g_{j}(0)\equiv 0 $, $ h_{j}(0)\equiv 0 $ and are globally Lipschitz functions with Lipschitz constants $ \alpha_{j}$, $ \beta_{j} $, $ \gamma_{j} $. That is for any $x, y\in C(\mathbb{R}^{+}, {\mathbb{R}}^{n})$, $t\geq\vartheta$, $ j=1, 2, 3, \cdots, n $, 
	\begin{eqnarray*}
		\left|f_j(x_j(t))-f_j(y_j(t))\right|&\leq& \alpha_j \left|x_j(t)-y_j(t)\right|, \\
		\left|g_j(x_j(t))-g_j(y_j(t))\right|&\leq& \beta_j \left|x_j(t)-y_j(t)\right|, \\
		\left|h_j(x_j(t))-h_j(y_j(t))\right|&\leq& \gamma_j \left|x_j(t)-y_j(t)\right|. 
	\end{eqnarray*}
	\item[{\rm{(A3)}}] For $i=1, 2, \cdots, n$, $k=1, 2, 3, \cdots$, the mapping $I_{ik}(t_k, (F(\cdot))_i(t_k))$ satisfies $ I_{ik}(t_k, (F(0))_i(t_k))\equiv 0 $ and is a globally Lipschitz function with a Lipschitz constant $ p_{ik} $. That is for any $x, y\in C(\mathbb{R}^{+}, {\mathbb{R}}^{n})$, $i=1, 2, \cdots, n$, $k=1, 2, 3, \cdots$, 
	\begin{eqnarray*}
		\left|I_{ik}(t_k, (F(x))_i(t_k))-I_{ik}(t_k, (F(y))_i(t_k))\right|&\leq& p_{ik} \left\|x(t_k)-y(t_k)\right\|. 
	\end{eqnarray*}
\end{enumerate}
\section{Asymptotic Stability}
In this section, we study asymptotic stability of the system (\ref{main equation}) by employing the fixed point method. 

Let $\mathcal{H} =\mathcal{H}_1\times \cdots\times \mathcal{H}_n$, and let $\mathcal{H}_i (i=1, \cdots, n)$ be the space consisting of function $\phi_i(t):[\vartheta, \infty)\rightarrow\mathbb{R}$, where $\phi_{i}(t)$ satisfies the following:
\begin{enumerate}[(1)]
	\item $\phi_i(s)=\varphi_i(s)$ on $s\in[\vartheta, 0]$;
	\item $\phi_{i}(t)$ is continuous on $t\ne t_k(k=1, 2, \cdots)$;
	\item $\lim_{t\rightarrow t_k^{-}}\phi_{i}(t)$ and $\lim_{t\rightarrow t_k^{+}}\phi_{i}(t)$ exist, furthermore, $\lim_{t\rightarrow t_k^{+}}\phi_{i}(t)=\phi_i(t_k)$ for $k=1, 2, \cdots$;
	\item $\phi_{i}(t)\rightarrow0$ as $t\rightarrow\infty$, 
\end{enumerate}
where $t_k\left(k=1, 2, \cdots\right)$ and $\varphi_i(s)\left(s\in [-\vartheta, 0)\right)$ are defined as shown in Section 2. Also, if we define the metric as 
$d(\phi, \psi)=\sum_{i=1}^{n}\sup_{t\geq\vartheta}|\phi_{i}(t)-\psi_{i}(t)|, $
then $\mathcal{H}$ is a complete metric space. 
\begin{theorem}
	Consider the nonlinear impulsive neutral delayed system {\rm{(\ref{main equation})}}. Suppose that the assumptions {\rm{ (A1)-(A3)}} hold and the following conditions are satisfied:
	\begin{enumerate}[(i)]
		\item the delay $r(t)$ is bounded by a positive constant $\mu$;
		\item there exist constants $p_i$ such that $p_{ik}\leq p_i(t_k-t_{k-1})$ for $i=1, 2, \cdots, n$, and $k=1, 2, \cdots$;
		\item there exist constants $\eta_i>0$ such that $v_i(t)>\eta_i$, $t\in \mathbb{R}^{+}$ for $i=1, 2, \cdots, n$;
		\item and such that 
		\begin{eqnarray*}
			&&\sum\limits_{i=1}^{n}\Bigg[\max_{j=1, \cdots, n}\sup_{t\geq\vartheta}|q_{ij}(t)|+\Bigg(\max_{j=1, \cdots, n}\sup_{t\geq\vartheta}|\overline{c}_{ij}(t)|+\max_{j=1, \cdots, n}\sup_{t\geq\vartheta}|a_{ij}(t)\alpha_j|+\max_{j=1, \cdots, n}\sup_{t\geq\vartheta}|b_{ij}(t)\beta_j|\\
			&&\quad+\max_{j=1, \cdots, n}\sup_{t\geq\vartheta}| w_{ij}(t)\mu\gamma_j|+\max_{j=1, \cdots, n}\sup_{t\geq\vartheta}|q_{ij}(t)v_i(t)|+|p_i|\Bigg)\times\sup\limits_{t\geq\vartheta}\int_{0}^{t}e^{-\int_{s}^{t}v_i(u)du}ds\Bigg]\triangleq \rho<1. 
		\end{eqnarray*}
	\end{enumerate}
	Then the trivial solution $x=0$ of system {\rm{(\ref{main equation})}} is asymptotically stable. 
\end{theorem}

\begin{proof}
	The following proof mainly relies on the Banach fixed point theorem, which will be divided into four steps. \\
	
	\textbf{Step 1}. Define an operator $\pi$ by 
	$$\pi(x)(t)=\left(\pi(x_1)(t), \cdots, \pi(x_n)(t)\right)^{T}, $$
	for $x(t)=\left(x_1(t), \cdots, x_n(t)\right)^{T}\in\mathcal{H}$, where $\pi(x_i)(t):[\vartheta, \infty)\rightarrow\mathbb{R}(i=1, 2, \cdots, n)$ obeys the rules as follows:
	\begin{eqnarray}\label{ASenq-1}
		\pi(x_i)(t)&=&\sum_{j=1}^{n}q_{ij}(t)x_{j}(t-\tau(t))+\left[\varphi_i(0)-\sum_{j=1}^{n}q_{ij}(0)\varphi_{j}(-\tau(0))\right]e^{-\int_{0}^{t}v_i(s)ds}+\sum_{0\leq t_k\leq t}I_{ik}(t_{k}, (Fx)_i(t_{k}))e^{-\int_{t_k}^{t}v_i(s)ds}\nonumber\\
		&&+\int_{0}^{t}e^{-\int_{s}^{t}v_i(u)du}\left[\sum_{j=1}^{n}\overline{c}_{ij}(s)x_{j}(s)+\sum_{j=1}^{n}a_{ij}(s)f_{j}(x_{j}(s))+\sum_{j=1}^{n}b_{ij}(s)g_{j}(x_{j}(s-\delta(s)))\right. \\
		&&\left. +\sum_{j=1}^{n}w_{ij}(s)\int_{s-r(s)}^{s}h_{j}(x_{j}(u))\, du-\sum_{j=1}^{n}v_i(s)q_{ij}(s)x_j(s-\tau(s))\right]ds, \nonumber
	\end{eqnarray}
	for $t\geq0$ and $\pi(x_i)(s)=\varphi_i(s)$ for $s\in[\vartheta, 0)$. 
	
	\textbf{Step 2}. We prove $\pi(\mathcal{H})\subseteq\mathcal{H}$. Choose $x_i(t)\in\mathcal{H}_i(i=1, 2, \cdots, n)$, it is necessary to testify $\pi (x_i)(t)\subseteq\mathcal{H}_i$. 
	First, since $\pi(x_i)(s)=\varphi(s)$ on $s\in [\vartheta, 0]$ and $\varphi(s)\in C([\vartheta, 0], \mathbb{R})$, we know $\pi(x_i)(s)$ is continuous on $s\in [\vartheta, 0]$. For a fixed time $t>0$, it follows from (\ref{ASenq-1}) that
	\begin{eqnarray}\label{con}
	\pi(x_i)(t+r)-\pi(x_i)(t)=R_1(t)+R_2(t)+R_3(t)+R_4(t), 
	\end{eqnarray}
	where
	\begin{eqnarray*}
		R_1(t)&=&\sum_{j=1}^{n}q_{ij}(t+r)x_{j}(t+r-\tau(t+r))-\sum_{j=1}^{n}q_{ij}(t)x_{j}(t-\tau(t)), \\
		R_2(t)&=&\left[\varphi_i(0)-\sum_{j=1}^{n}q_{ij}(0)\varphi_{j}(-\tau(0))\right]e^{-\int_{0}^{t+r}v_i(s)ds}-\left[\varphi_i(0)-\sum_{j=1}^{n}q_{ij}(0)\varphi_{j}(-\tau(0))\right]e^{-\int_{0}^{t}v_i(s)ds}, \\
		R_3(t)&=&\sum_{0\leq t_k\leq t+r}I_{ik}(t_{k}, (Fx)_i(t_{k}))e^{-\int_{t_k}^{t+r}v_i(s)ds}-\sum_{0\leq t_k\leq t}I_{ik}(t_{k}, (Fx)_i(t_{k}))e^{-\int_{t_k}^{t}v_i(s)ds}, \\
		R_4(t)&=&\int_{0}^{t+r}e^{-\int_{s}^{t+r}v_i(u)du}\left[\sum_{j=1}^{n}\overline{c}_{ij}(s)x_{j}(s)+\sum_{j=1}^{n}a_{ij}(s)f_{j}(x_{j}(s))+\sum_{j=1}^{n}b_{ij}(s)g_{j}(x_{j}(s-\delta(s)))\right. \\
		&&\left. + \sum_{j=1}^{n}w_{ij}(s)\int_{s-r(s)}^{s}h_{j}(x_{j}(u))\, du-\sum_{j=1}^{n}v_i(s)q_{ij}(s)x_j(s-\tau(s))\right]ds\\
		&&-\int_{0}^{t}e^{-\int_{s}^{t}v_i(u)du}\left[\sum_{j=1}^{n}\overline{c}_{ij}(s)x_{j}(s)+\sum_{j=1}^{n}a_{ij}(s)f_{j}(x_{j}(s))+\sum_{j=1}^{n}b_{ij}(s)g_{j}(x_{j}(s-\delta(s)))\right. \\
		&&\left. + \sum_{j=1}^{n}w_{ij}(s)\int_{s-r(s)}^{s}h_{j}(x_{j}(u))\, du-\sum_{j=1}^{n}v_i(s)q_{ij}(s)x_j(s-\tau(s))\right]ds. 
	\end{eqnarray*}
	
	It is clear that $x_{i}(t)$ is continuous on $t \neq t_{k}(k=1, 2, \cdots)$. Moreover, $\lim _{t \rightarrow t_{k}^{-}} x_{i}(t)$ and $\lim _{t \rightarrow t_{k}^{+}} x_{i}(t)$ exist, and $\lim _{t \rightarrow t_{k}^{+}} x_{i}(t)=x_{i}\left(t_{k}\right)$. we can check that $R_{i}(t) \rightarrow 0 $ as $r \rightarrow 0$ on $t \neq t_{k} $ $(i=1, 2, 3, 4)$, so $\pi\left(x_{i}\right)(t)$ is continuous on the fixed time $t \neq t_{k} (k=1, 2, \cdots)$ . 
	
	On the other hand, as $t=t_{k}(k=1, 2, \cdots)$ in (\ref{con}), it is not difficult to find that $R_{i}(t) \rightarrow 0$ as $r \rightarrow 0$ for $i=1, 2, 3, 4$. Furthermore, let $r>0$ be small enough, we derive
	\begin{eqnarray*}
		R_{3}(t)&= & e^{-a_{i}\left(t_{k}+r\right)} \sum_{0\leq t_{m }\leq\left(t_{k}+r\right)} I_{im}(t_m, (Fx)_i(t_m)) e^{a_{i} t_{m}}-e^{-a_{i} t_{k}} \sum_{0\leq t_{m}\leq t_{k}} I_{im}(t_m, (Fx)_i(t_m)) e^{a_{i} t_{m}} \\
		&= & \left(e^{-a_{i}\left(t_{k}+r\right)}-e^{-a_{i} t_{k}}\right) \sum_{0\leq t_{m}\leq t_{k}}\left\{I_{im}(t_m, (Fx)_i(t_m)) e^{a_{i} t_{m}}\right\}, 
	\end{eqnarray*}
	which implies $\lim _{r \rightarrow 0^{+}} R_{3}(t)=0$ as $t=t_{k}$ . 
	
	While letting $r<0$ tend to zero gives
	\begin{eqnarray*}
		R_{3}(t)&= & e^{-a_{i}\left(t_{k}+r\right)} \sum_{0\leq t_{m}\leq\left(t_{k}+r\right)} I_{im}(t_m, (Fx)_i(t_m)) e^{a_{i} t_{m}}-e^{-a_{i} t_{k}} \sum_{0\leq t_{m}\leq t_{k}} I_{im}(t_m, (Fx)_i(t_m)) e^{a_{i} t_{m}}\\
		&= & \left(e^{-a_{i}\left(t_{k}+r\right)}-e^{-a_{i} t_{k}}\right) \sum_{0\leq t_{m}\leq(t_{k}+r)}I_{im}(t_m, (Fx)_i(t_m)) e^{a_{i} t_{m}}-I_{i k}(t_k, (Fx)_i(t_{k})), 
	\end{eqnarray*}
which yields $\lim _{r \rightarrow 0^{-}} R_{3}(t)=-I_{i k}(t_k, (Fx)_i(t_{k})), $ as $t=t_{k}$ . 
	According to the above discussion, we find that $\pi\left(x_{i}\right)(t) : [\vartheta, \infty) \rightarrow \mathbb{R}$ is continuous on $t \neq t_{k}(k=1, 2, \cdots)$, moreover, $\lim _{t \rightarrow t_{k}^{-}} \pi\left(x_{i}\right)(t) and \lim _{t \rightarrow t_{k}^{+}} \pi\left(x_{i}\right)(t)$ exist, $\pi\left(x_{i}\right)\left(t_{k}\right) = \lim _{t \rightarrow t_{k}^{+}} \pi\left(x_{i}\right)(t)$ . 
	
	Next, we prove $\pi\left(x_{i}\right)(t) \rightarrow 0$ as $t \rightarrow \infty$ . For convenience, denote
	\begin{eqnarray*}
		\pi(x_i)(t)=S_1(t)+S_2(t)+S_3(t)+S_4(t), 
	\end{eqnarray*}
	where
	\begin{eqnarray}\label{S}
		S_1(t)&=&\sum_{j=1}^{n}q_{ij}(t)x_{j}(t-\tau(t)),\quad  S_2(t)=\left[\varphi_i(0)-\sum_{j=1}^{n}q_{ij}(0)x_{j}(-\tau(0))\right]e^{-\int_{0}^{t}v_i(s)\, ds}, \nonumber\\
		S_3(t)&=&\sum_{0\leq t_k\leq t}I_{ik}(t_{k}, (Fx)_i(t_{k}))e^{-\int_{t_k}^{t}v_i(s)\, ds}, \\
		S_4(t)&=&\int_{0}^{t}e^{-\int_{s}^{t}v_i(u)\, du}\left[\sum_{j=1}^{n}\overline{c}_{ij}(s)x_{j}(s)+\sum_{j=1}^{n}a_{ij}(s)f_{j}(x_{j}(s))+\sum_{j=1}^{n}b_{ij}(s)g_{j}(x_{j}(s-\delta(s)))\right. \nonumber\\
		&&\left. +\sum_{j=1}^{n}w_{ij}(s)\int_{s-r(s)}^{s}h_{j}(x_{j}(u))\, du-\sum_{j=1}^{n}v_i(s)q_{ij}(s)x_j(s-\tau(s))\right]ds. \nonumber
	\end{eqnarray}
	
	Since $t-\tau(t) \rightarrow \infty$ as $t \rightarrow \infty$ , we get $\lim _{t \rightarrow \infty} x_{j}\left(t-\tau(t)\right)=0 $. Then for any $\varepsilon>0$, there also exists a $T_{j}>0$ such that $t\geq T_{j}$ implies $\left|x_{j}\left(t-\tau(t)\right)\right|<\varepsilon $. Select $\overline{T}=\max _{j=1, \cdots, n}\left\{T_{j}\right\} $. It follows that
	\begin{eqnarray*}
		S_1(t)&\leq&\varepsilon\sum_{j=1}^{n}\sup_{ t\geq\vartheta}|q_{ij}(t)|, 
	\end{eqnarray*}
	which implies $S_1(t)\rightarrow0$ as $t\rightarrow\infty$. 
	
	By condition (iii), we have 
$$\int_{0}^{t}v_i(s)\, ds\rightarrow\infty\, \text{as}\, t\rightarrow\infty,$$
	which leads to 
$$S_2(t)\leq\left[\varphi_i(0)-\sum_{j=1}^{n}q_{ij}(0)\varphi_{j}(-\tau(0))\right]\varepsilon,$$
 which implies $S_2(t)\rightarrow0$ as $t\rightarrow\infty$. 
	
	Then for any $\varepsilon>0$, there exists a nonimpulsive point $T_i>0$ such that $t>T_i$ implies $|x_i(t)|<\varepsilon$. Then
	\begin{eqnarray*}
		S_3(t)&\leq&\sum_{0\leq t_k\leq T_i}p_i(t_k-t_{k-1}) e^{-\int_{t_k}^{t}v_i(s)ds}\sum_{j=1}^{n}|x_j(t_k)|+\sum_{T_i<t_k\leq t}p_i(t_k-t_{k-1}) e^{-\int_{t_k}^{t}v_i(s)ds}\sum_{j=1}^{n}|x_j(t_k)|\\
		&\leq&e^{-\int_{0}^{t}v_i(s)ds}\sum_{0\leq t_k\leq T_i}p_i(t_k-t_{k-1}) e^{-\int_{t_k}^{0}v_i(s)ds}\sum_{j=1}^{n}|x_j(t_k)|+np_i\varepsilon\frac{1}{\eta_i}-np_i\varepsilon\frac{1}{\eta_i}e^{-\eta_it+\eta_iT_i}, 
	\end{eqnarray*}
	which implies $S_3(t)\rightarrow0$ as $t\rightarrow\infty$. 
	
	Since $ t-\tau(t)\rightarrow \infty $, $ t-\delta(t)\rightarrow \infty $ and $ t-r(t)\rightarrow \infty $ as $ t\rightarrow \infty $, we get $\lim _{t \rightarrow \infty} x_{j}\left(t-\tau(t)\right)=0 $, $\lim _{t \rightarrow \infty} x_{j}\left(t-\delta(t)\right)=0 $, $\lim _{t \rightarrow \infty} x_{j}\left(t-r(t)\right)=0 $. Then for any $\varepsilon>0$, there also exists a $T_{j}>0$ such that $t\geq T_{j}$ implies $\left|x_{j}\left(t-\tau(t)\right)\right|<\varepsilon $, $\left|x_{j}\left(t-\delta(t)\right)\right|<\varepsilon $, $\left|x_{j}\left(t-r(t)\right)\right|<\varepsilon $. Select $\overline{T}=\max _{j=1, \cdots, n}\left\{T_{j}\right\} $, we have
	\begin{eqnarray*}
		S_4(t)&\leq&\int_{0}^{\bar{T}}e^{-\int_{s}^{t}v_i(u)du}\left[\sum_{j=1}^{n}\overline{c}_{ij}(s)x_{j}(s)+\sum_{j=1}^{n}a_{ij}(s)f_{j}(x_{j}(s))+\sum_{j=1}^{n}b_{ij}(s)g_{j}(x_{j}(s-\delta(s)))\right. \\
		&&\left. +\sum_{j=1}^{n}w_{ij}(s)\int_{s-r(s)}^{s}h_{j}(x_{j}(u))\, du-\sum_{j=1}^{n}v_i(s)q_{ij}(s)x_j(s-\tau(s))\right]ds\\
		&&+\int_{\bar{T}}^{t}e^{-\int_{s}^{t}v_i(u)du}\left[\sum_{j=1}^{n}\overline{c}_{ij}(s)x_{j}(s)+\sum_{j=1}^{n}a_{ij}(s)f_{j}(x_{j}(s))+\sum_{j=1}^{n}b_{ij}(s)g_{j}(x_{j}(s-\delta(s)))\right. \\
		&&\left. +\sum_{j=1}^{n}w_{ij}(s)\int_{s-r(s)}^{s}h_{j}(x_{j}(u))\, du-\sum_{j=1}^{n}v_i(s)q_{ij}(s)x_j(s-\tau(s))\right]ds\\
\end{eqnarray*}
\begin{eqnarray*}
		&\leq&e^{-\eta_i t}\int_{0}^{\bar{T}}e^{-\int_{s}^{0}v_i(u)du}\left[\sum_{j=1}^{n}\overline{c}_{ij}(s)x_{j}(s)+\sum_{j=1}^{n}a_{ij}(s)f_{j}(x_{j}(s))+\sum_{j=1}^{n}b_{ij}(s)g_{j}(x_{j}(s-\delta(s)))\right. \\
		&&\left. +\sum_{j=1}^{n}w_{ij}(s)\int_{s-r(s)}^{s}h_{j}(x_{j}(u))\, du-\sum_{j=1}^{n}v_i(s)q_{ij}(s)x_j(s-\tau(s))\right]ds\\
		&&+ \frac{\varepsilon}{\eta_i}\left[\sum_{j=1}^{n}\sup_{ s\geq\vartheta}|\overline{c}_{ij}(s)|+\sum_{j=1}^{n}\sup_{ s\geq\vartheta}|a_{ij}(s)\alpha_j|+\sum_{j=1}^{n}\sup_{ s\geq\vartheta}|b_{ij}(s)\beta_j|+\sum_{j=1}^{n}\sup_{ s\geq\vartheta}|w_{ij}(s)\mu\gamma_j|\right. \\
		&&\left. +\sum_{j=1}^{n}\sup_{ s\geq\vartheta}|v_i(s)q_{ij}(s)|\right]\left(1-e^{\eta_i(\bar{T}-t)}\right), \\
	\end{eqnarray*}
	which implies $S_4(t)\rightarrow0$ as $t\rightarrow\infty$. 
	
	Therefore, we deduce $\pi(x_i)(t)\rightarrow0$ as $t\rightarrow\infty$ for $i=1, \cdots, n$. We conclude that $\pi(x_i)(t)\subset\mathcal{H}(i=1, \cdots, n)$ which means $\pi(\mathcal{H})\subseteq\mathcal{H}$. 
	
	\textbf{Step 3}. In order to use the Banach fixed point theorem, we need to prove $\pi$ is a contraction mapping. For any $y=\left(y_{1}(t), \cdots, y_{n}(t)\right)^{T} \in \mathcal{H}$ and $z=\left(z_{1}(t), \cdots, z_{n}(t)\right)^{T} \in \mathcal{H}$ , we have
	$$\pi\left(y_{i}\right)(t)-\pi\left(z_{i}\right)(t)= T_{1}(t)+T_{2}(t)+T_{3}(t), $$
	where
	\begin{eqnarray*}
		T_1(t)&=&\sum_{j=1}^{n}q_{ij}(t)y_{j}(t-\tau(t))-\sum_{j=1}^{n}q_{ij}(t)z_{j}(t-\tau(t)), \\
		T_2(t)&=&\sum_{0\leq t_k\leq t}I_{ik}(t_{k}, (Fx)_i(t_{k}))e^{-\int_{t_k}^{t}v_i(s)ds}-\sum_{0\leq t_k\leq t}I_{ik}(t_{k}, (Fx)_i(t_{k}))e^{-\int_{t_k}^{t}v_i(s)ds}, \\
		T_3(t)&=&\int_{0}^{t}e^{-\int_{s}^{t}v_i(u)du}\left[\sum_{j=1}^{n}\overline{c}_{ij}(s)y_{j}(s)+\sum_{j=1}^{n}a_{ij}(s)f_{j}(y_{j}(s))+\sum_{j=1}^{n}b_{ij}(s)g_{j}(y_{j}(s-\tau(s)))\right. \\
		&&\left. +\sum_{j=1}^{n}w_{ij}(s)\int_{s-r(s)}^{s}h_{j}(y_{j}(u))\, du-\sum_{j=1}^{n}v_i(s)q_{ij}(s)y_j(s-\delta(s))\right]ds\\
		&&-\int_{0}^{t}e^{-\int_{s}^{t}v_i(u)du}\left[\sum_{j=1}^{n}\overline{c}_{ij}(s)z_{j}(s)+\sum_{j=1}^{n}a_{ij}(s)f_{j}(z_{j}(s))+\sum_{j=1}^{n}b_{ij}(s)g_{j}(z_{j}(s-\tau(s)))\right. \\
		&&\left. +\sum_{j=1}^{n}w_{ij}(s)\int_{s-r(s)}^{s}h_{j}(z_{j}(u))\, du-\sum_{j=1}^{n}v_i(s)q_{ij}(s)z_j(s-\delta(s))\right]ds. 
	\end{eqnarray*}
	
	Note that 
	\begin{eqnarray*}
		|T_1(t)|&=&\left|\sum_{j=1}^{n}q_{ij}(t)\left[y_j(t-\tau(t))-z_j(t-\tau(t))\right]\right|
		\leq\max_{j=1, \cdots, n}\sup_{t\geq\vartheta}\left|q_{ij}(t)\right|\sum_{j=1}^{n}\left[\sup_{t\geq\vartheta}\left|y_j(t)-z_j(t)\right|\right], \\
		|T_2(t)|&=&\left|\sum_{0\leq t_k\leq t}\left[I_{ik}(t_{k}, (Fy)_i(t_{k}))-I_{ik}(t_{k}, (Fz)_i(t_{k}))\right]e^{-\int_{t_k}^{t}v_i(s)ds}\right|\\
		&\leq&\sum_{0\leq t_k\leq t}\left[|p_{ik}|e^{-\int_{t_k}^{t}v_i(s)ds}\sum_{j=1}^{n}\left|y_j(t_k)-z_j(t_k)\right|\right]\\
		&\leq&|p_i|\int_{0}^{t}e^{-\int_{s}^{t}v_i(u)du}ds\sum_{j=1}^{n}\left[\sup_{t\geq\vartheta}\left|y_j(t)-z_j(t)\right|\right], 
	\end{eqnarray*}
	\begin{eqnarray*}
		|T_3(t)|&\leq&\left[\max_{j=1, \cdots, n}\sup_{t\geq\vartheta}|\overline{c}_{ij}(t)|+\max_{j=1, \cdots, n}\sup_{t\geq\vartheta}|a_{ij}(t)\alpha_j|+\max_{j=1, \cdots, n}\sup_{t\geq\vartheta}|b_{ij}(t)\beta_j|+\max_{j=1, \cdots, n}\sup_{t\geq\vartheta}| w_{ij}(t)\mu\gamma_j|\right. \\
		&&\left. +\max_{j=1, \cdots, n}\sup_{t\geq\vartheta}|q_{ij}(t)v_i(t)|\right]\int_{0}^{t}e^{-\int_{s}^{t}v_i(u)du}ds\sum_{j=1}^{n}\left[\sup_{t\geq\vartheta}\left|y_j(t)-z_j(t)\right|\right]. 
	\end{eqnarray*}
	
	It follows that
	\begin{eqnarray*}
		&&\left|\pi\left(y_{i}\right)(t)-\pi\left(z_{i}\right)(t)\right|\\ 
		&&\quad\leq\left\{\max_{j=1, \cdots, n}\sup_{t\geq\vartheta}\left|q_{ij}(t)\right|+\left[\max_{j=1, \cdots, n}\sup_{t\geq\vartheta}|\overline{c}_{ij}(t)|+\max_{j=1, \cdots, n}\sup_{t\geq\vartheta}|a_{ij}(t)\alpha_j|+\max_{j=1, \cdots, n}\sup_{t\geq\vartheta}|b_{ij}(t)\beta_j|\right. \right. \\
		&&\quad\quad\left. \left. +\max_{j=1, \cdots, n}\sup_{t\geq\vartheta}| w_{ij}(t)\mu\gamma_j|+\max_{j=1, \cdots, n}\sup_{t\geq\vartheta}|q_{ij}(t)v_i(t)|+|p_i|\right]\int_{0}^{t}e^{-\int_{s}^{t}v_i(u)du}ds\right\}\sum_{j=1}^{n}\left[\sup_{t\geq\vartheta}\left|y_j(t)-z_j(t)\right|\right], 
	\end{eqnarray*}
	which implies
	\begin{eqnarray*}
		&&\sum_{i=1}^{n}\sup_{t\geq\vartheta}\left|\pi\left(y_{i}\right)(t)-\pi\left(z_{i}\right)(t)\right|\\ 
		&&\quad\leq\left\{\sum_{i=1}^{n}\left[\max_{j=1, \cdots, n}\sup_{t\geq\vartheta}\left|q_{ij}(t)\right|+\left(\max_{j=1, \cdots, n}\sup_{t\geq\vartheta}|\overline{c}_{ij}(t)|+\max_{j=1, \cdots, n}\sup_{t\geq\vartheta}|a_{ij}(t)\alpha_j|+\max_{j=1, \cdots, n}\sup_{t\geq\vartheta}|b_{ij}(t)\beta_j|\right. \right. \right. \\
		&&\quad\quad\left. \left. \left. +\max_{j=1, \cdots, n}\sup_{t\geq\vartheta}| w_{ij}(t)\mu\gamma_j|+\max_{j=1, \cdots, n}\sup_{t\geq\vartheta}|q_{ij}(t)v_i(t)|+|p_i|\right)\sup_{t\geq\vartheta}\int_{0}^{t}e^{-\int_{s}^{t}v_i(u)du}ds\right]\right\}\sum_{j=1}^{n}\left[\sup_{t\geq\vartheta}\left|y_j(t)-z_j(t)\right|\right]. \\	
	\end{eqnarray*}
	
	In view of condition (iv), we see $\pi$ is a contraction mapping. By the Banach fixed point theorem, we obtain that $\pi$ has a unique fixed point $x(t)$ in $\mathcal{H}$, which is a solution of (\ref{main equation}) with $x(t)=\varphi$ as $t\in[\vartheta, 0]$ and $x(t)\rightarrow 0$ as $t\rightarrow \infty$. 

	\textbf{Step 4}. To obtain the asymptotic stability, we still need to prove that the trivial solution $x=0$ is stable. 
	
	For any $\varepsilon>0$ , from condition (iv), we can find $\delta$ satisfying $0<\delta<\varepsilon$ such that $\delta+\rho \varepsilon \leq \varepsilon $. 
	Let $\|\varphi\|<\delta^{\prime}$ . 
	
	According to the above discussion, we know that there exists a unique solution 
	$$x(t , s, \varphi)=\left(x_{1}\left(t , s, \varphi_{1}\right), \cdots, x_{n}\left(t , s, \varphi_{n}\right)\right)^{T}. $$
	
	Moreover, let
	$$x_{i}(t)=\pi\left(x_{i}\right)(t)=S_{1}(t)+S_{2}(t)+S_{3}(t)+S_{4}(t), \quad t \geq 0, $$
	where $S_1(t)$, $S_2(t)$, $S_3(t)$, $S_4(t)$ are denoted by (\ref{S}). 
	
	Suppose there exists $t^{*}>0$ such that $\left\|x\left(t^{*} , s, \varphi\right)\right\|=\varepsilon$ and $\|x(t , s, \varphi)\|<\varepsilon $ for $ 0 \leq t<t^{*} $, we have
	$$|x_i(t^{*})|\leq|S_1(t^{*})|+|S_2(t^{*})|+|S_3(t^{*})|+|S_4(t^{*})|, $$
	then
	\begin{eqnarray*}
		|S_1(t^{*})|&\leq&\max_{j=1, \cdots, n}\sup_{t\geq\vartheta}|q_{ij}(t)|\left[\sum_{j=1}^{n}\left|x_{j}(t^{*}-\tau(t^{*}))\right|\right]
		<\max_{j=1, \cdots, n}\sup_{t\geq\vartheta}|q_{ij}(t)|\varepsilon, 
	\end{eqnarray*}
	\begin{eqnarray*}
		|S_2(t^{*})|&\leq&\left|\varphi_i(0)\right|e^{-\int_{0}^{t^{*}}v_i(s)ds}+\max_{j=1, \cdots, n}\left|q_{ij}(0)\right|e^{-\int_{0}^{t^{*}}v_i(s)ds}\sum_{j=1}^{n}\left|x_{j}(-\tau(0))\right|\\
		&<&\left|\varphi_i(0)\right|e^{-\int_{0}^{t^{*}}v_i(s)ds}+\max_{j=1, \cdots, n}\left|q_{ij}(0)\right|e^{-\int_{0}^{t^{*}}v_i(s)ds}\delta^{\prime}, 
	\end{eqnarray*}
	\begin{eqnarray*}
		|S_3(t^{*})|\leq\sum_{0\leq t_k\leq t^{*}}|p_i|(t_k-t_{k-1}) e^{-\int_{t_k}^{t^{*}}v_i(u)\, du}\sum_{j=1}^{n}|x_j(t_k)|
		&\leq &|p_i|\int_{0}^{t^{*}}e^{-\int_{s}^{t^{*}}v_i(u)\, du}ds\sup_{0\leq t_k \leq t^*}\sum_{j=1}^{n}|x_j(t_k)|\\
		&<&\varepsilon|p_i|\sup_{t\geq\vartheta}\int_{0}^{t}e^{-\int_{s}^{t}v_i(u)\, du}ds, 
	\end{eqnarray*}
	\begin{eqnarray*}
		|S_4(t^{*})|
		&<&\varepsilon\left[\max_{j=1, \cdots, n}\sup_{t\geq\vartheta}|\overline{c}_{ij}(t)|+\max_{j=1, \cdots, n}\sup_{t\geq\vartheta}|a_{ij}(t)\alpha_j|+\max_{j=1, \cdots, n}\sup_{t\geq\vartheta}|b_{ij}(t)\beta_j|\right. \\
		&&\left. +\max_{j=1, \cdots, n}\sup_{t\geq\vartheta}| w_{ij}(t)\mu\gamma_j|+\max_{j=1, \cdots, n}\sup_{t\geq\vartheta}|q_{ij}(t)v_i(t)|ds\right]\sup_{t\geq\vartheta}\int_{0}^{t}e^{-\int_{s}^{t}v_i(u)\, du}ds. 
	\end{eqnarray*}
	
	Hence, we have 
	\begin{eqnarray*}
		\|x(t^{*} , s, \varphi)\|
		=\sum_{i=1}^{n}\left|x_i(t^*)\right|
		&<&\sum_{i=1}^{n}\left[\left|\varphi_i(0)\right|e^{-\int_{0}^{t^{*}}v_i(s)ds}+\max_{j=1, \cdots, n}\left|q_{ij}(0)\right|e^{-\int_{0}^{t^{*}}v_i(s)ds}\delta^{\prime}\right]+\rho\varepsilon\\
		&<&\left[\left(1+\sum_{i=1}^{n}\max_{j=1, \cdots, n}\left|q_{ij}(0)\right|\right)\max_{i=1, \cdots, n}\sup_{t\geq \vartheta}e^{-\int_{0}^{t^{*}}v_i(s)ds}\right]\delta^{\prime}+\rho\varepsilon. 
	\end{eqnarray*}
	
Based on the definition of $\delta$ above, we can choose $$\delta=\left[\left(1+\sum_{i=1}^{n}\max_{j=1, \cdots, n}\left|q_{ij}(0)\right|\right)\max_{i=1, \cdots, n}\sup_{t\geq \vartheta}e^{-\int_{0}^{t^{*}}v_i(s)ds}\right]\delta^{\prime}$$ that is sufficiently small so that the above equation is less than $\varepsilon$. 
	
	This contradicts the assumption of $\|x(t^{*}, 0, \varphi)\|=\varepsilon$. Therefore, $\|x(t, 0, \varphi)\|<\varepsilon$ holds for all $t\ge0$. 
	This completes the proof. 
\end{proof}
Consider the case when there are no impulsive effects, the system {\rm(\ref{main equation})} reduced to the following
\begin{eqnarray}\label{main equation2}
&&\left\{
\begin{array}{ll}
d\left[x_{i}(t)-\sum_{j=1}^{n}q_{ij}(t)x_{j}(t-\tau(t))\right] \\
\quad =\left[\sum_{j=1}^{n}c_{ij}(t)x_{j}(t)+\sum_{j=1}^{n}a_{ij}(t)f_{j}(x_{j}(t))+\sum_{j=1}^{n}b_{ij}(t)g_{j}(x_{j}(t-\delta(t)))
+\sum_{j=1}^{n}w_{ij}(t)\int_{t-r(t)}^{t}h_j(x_{j}(s))\, ds\right]\, dt, \\
\qquad\qquad\qquad\qquad\qquad\qquad\qquad\qquad\qquad\qquad\qquad\qquad\qquad\qquad\qquad\qquad\qquad t\geq 0, \\
&\\
x(t)=\varphi(t), \quad \quad t\in[\vartheta, 0]. 
\end{array}
\right. 
\end{eqnarray}
\begin{corollary}\label{cor1}
	Consider the nonlinear neutral delayed system {\rm{(\ref{main equation2})}}. Suppose that the assumptions {\rm{ (A1)-(A3)}} hold and the following conditions are satisfied:
	\begin{enumerate}[(i)]
		\item the delay $r(t)$ is bounded by a positive constant $\mu$;
		\item there exist constants $\eta_i>0$ such that $v_i(t)>\eta_i$, $t\in \mathbb{R}^{+}$ for $i=1, 2, \cdots, n$;
		\item and such that 
		\begin{eqnarray*}
			\sum\limits_{i=1}^{n}\Bigg[\max_{j=1, \cdots, n}\sup_{t\geq\vartheta}|q_{ij}(t)|+\Bigg(\max_{j=1, \cdots, n}\sup_{t\geq\vartheta}|\overline{c}_{ij}(t)|+\max_{j=1, \cdots, n}\sup_{t\geq\vartheta}|a_{ij}(t)\alpha_j|+\max_{j=1, \cdots, n}\sup_{t\geq\vartheta}|b_{ij}(t)\beta_j|\\
			+\max_{j=1, \cdots, n}\sup_{t\geq\vartheta}| w_{ij}(t)\mu\gamma_j|+\max_{j=1, \cdots, n}\sup_{t\geq\vartheta}|q_{ij}(t)v_i(t)|\Bigg)\times\sup\limits_{t\geq\vartheta}\int_{0}^{t}e^{-\int_{s}^{t}v_i(u)du}ds\Bigg]<1. \qquad
		\end{eqnarray*}
	\end{enumerate}
	Then the trivial solution $x=0$ of system {\rm{(\ref{main equation2})}} is asymptotically stable. 

\end{corollary}
\begin{remark}
	 Chen et al.{\rm{\cite{GOS1}}} has studied asymptotic stability of a special case of the system {\rm{(\ref{main equation2})}}. The system studied in {\rm{\cite{GOS1}}} has no neutral term, and the coefficient are constants. Our results in Corollary {\rm{\ref{cor1}}} improve and extend the results in {\rm{\cite{GOS1}}}. 
\end{remark}

\section{Exponential Stability}
In this section, we study exponential stability of the system (\ref{main equation}) by employing the fixed point method. 

Let $\mathcal{H} =\mathcal{H}_1\times \cdots\times \mathcal{H}_n$, and let $\mathcal{H}_i (i=1, \cdots, n)$ be the space consisting of function $\phi_i(t):[\vartheta, \infty)\rightarrow\mathbb{R}$, where $\phi_{i}(t)$ satisfies the following:
\begin{enumerate}[(1)]
	\item $\phi_i(s)=\varphi_i(s)$ on $s\in[\vartheta, 0]$;
	\item $\phi_{i}(t)$ is continuous on $t\ne t_k(k=1, 2, \cdots)$;
	\item $\lim_{t\rightarrow t_k^{-}}\phi_{i}(t)$ and $\lim_{t\rightarrow t_k^{+}}\phi_{i}(t)$ exist, furthermore, $\lim_{t\rightarrow t_k^{+}}\phi_{i}(t)=\phi_i(t_k)$ for $k=1, 2, \cdots$;
	\item $e^{\lambda t}\phi_{i}(t)\rightarrow0$ as $t\rightarrow\infty$, where $\lambda<\min_{i=1, \cdots, n}\{\eta_i\}$, 
\end{enumerate}
where $t_k\left(k=1, 2, \cdots\right)$ and $\varphi_i(s)\left(s\in [-\vartheta, 0)\right)$ are defined as shown in Section 2. Also, if we define the metric as 
$d(\phi, \psi)=\sum_{i=1}^{n}\sup_{t\geq\vartheta}|\phi(t)-\psi(t)|, $
then $\mathcal{H}$ is a complete metric space. 
\begin{theorem}
	Consider the nonlinear impulsive neutral delayed system {\rm{(\ref{main equation})}}. Suppose that the assumptions {\rm{ (A1)-(A3)}} hold and the following conditions are satisfied:
	\begin{enumerate}[(i)]
		\item the delay $\delta(t)$, $\tau(t)$ and $r(t)$ are bounded by a positive constant $\mu$;
		\item there exist constants $p_i$ such that $p_{ik}\leq p_i(t_k-t_{k-1})$ for $i=1, 2, \cdots, n$, and $k=1, 2, \cdots$;
		\item there exist constants $\eta_i >0$ such that $v_i(t)>\eta _i$, $t\in \mathbb{R}^{+}$ for $i=1, 2, \cdots, n$;
		\item and such that 
		\begin{eqnarray*}
			&&\sum\limits_{i=1}^{n}\Bigg[\max_{j=1, \cdots, n}\sup_{t\geq\vartheta}|q_{ij}(t)|+\Bigg(\max_{j=1, \cdots, n}\sup_{t\geq\vartheta}|\overline{c}_{ij}(t)|+\max_{j=1, \cdots, n}\sup_{t\geq\vartheta}|a_{ij}(t)\alpha_j|+\max_{j=1, \cdots, n}\sup_{t\geq\vartheta}|b_{ij}(t)\beta_j|\\
			&&\quad+\max_{j=1, \cdots, n}\sup_{t\geq\vartheta}| w_{ij}(t)\mu\gamma_j|+\max_{j=1, \cdots, n}\sup_{t\geq\vartheta}|q_{ij}(t)v_i(t)|+|p_i|\Bigg)\times\sup\limits_{t\geq\vartheta}\int_{0}^{t}e^{-\int_{s}^{t}v_i(u)du}ds\Bigg]\triangleq \rho<1. 
		\end{eqnarray*}
	\end{enumerate}
	Then the trivial equilibrium $x=0$ of system {\rm{(\ref{main equation})}} is exponentially stable. 
	
\end{theorem}
\begin{proof}
	The following proof is based on the contraction mapping principle, which can be divided into three steps. \\
	
	\textbf{Step 1}. We define the following operator $\pi$ acting on $\mathcal{H}$, for $\overline{x}(t)=(x_1(t), \cdots, x_n(t))^{T}\in\mathcal{H}$:
	$$\pi(\overline{x})(t)=(\pi(x_1)(t), \cdots, \pi(x_n)(t))^{T}, $$
	where $\pi(x_i)(t):[\vartheta, \infty)\rightarrow\mathbb{R}(i=1, 2, \cdots, n)$ obeys the rules as follows:
	\begin{eqnarray}\label{ASenq}
		\pi(x_i)(t)&=&\sum_{j=1}^{n}q_{ij}(t)x_{j}(t-\tau(t))+\left[\varphi_i(0)-\sum_{j=1}^{n}q_{ij}(0)x_{j}(-\tau(0))\right]e^{-\int_{0}^{t}v_i(s)ds}+\sum_{0\leq t_k\leq t}I_{ik}(t_{k}, (Fx)_i(t_{k}))e^{-\int_{t_k}^{t}v_i(s)ds}\nonumber\\
		&&+\int_{0}^{t}e^{-\int_{s}^{t}v_i(u)du}\left[\sum_{j=1}^{n}\overline{c}_{ij}(s)x_{j}(s)+\sum_{j=1}^{n}a_{ij}(s)f_{j}(x_{j}(s))+\sum_{j=1}^{n}b_{ij}(s)g_{j}(x_{j}(s-\delta(s)))\right. \\
		&&\left. +\sum_{j=1}^{n}w_{ij}(s)\int_{s-r(s)}^{s}h_{j}(x_{j}(u))\, du-\sum_{j=1}^{n}v_i(s)q_{ij}(s)x_j(s-\tau(s))\right]ds, \nonumber
	\end{eqnarray}
	on $t\geq0$ and $\pi(x_i)(s)=\varphi_i(s)$ on $s\in[\vartheta, 0)$. 
	
	\textbf{Step 2}. Similar to the proof in Section 3, we know that $x_i(s)=\varphi_i(s)$ on $s\in[\vartheta, 0]$, 
	$x_{i}(t)$ is continuous on $t\ne t_k(k=1, 2, \cdots)$, 
	$\lim_{t\rightarrow t_k^{-}}x_{i}(t)$ and $\lim_{t\rightarrow t_k^{+}}x_{i}(t)$ exist, Furthermore, $\lim_{t\rightarrow t_k^{+}}x_{i}(t)=x_i(t_k)$ for $k=1, 2, \cdots$. 
	
	Next, we need to prove $e^{\lambda t}x_{i}(t)\rightarrow0$ as $t\rightarrow\infty$. 
	
	For convenience, denote
	$$\pi(x_i)(t)=S_1(t)+S_2(t)+S_3(t)+S_4(t), $$
	where $S_1(t)$, $S_2(t)$, $S_3(t)$, $S_4(t)$ are denoted by (\ref{S}). 
	
	Since $x_j(t)\in\mathcal{H}$ for $j=1, \cdots, n$, we know $\lim_{t\rightarrow\infty}e^{\lambda t}x_j(t)=0$. Then for any $\varepsilon>0$, there exists a $T_j>0$ such that $t\ge T_j$ implies $|e^{\lambda t}x_j(t)|<\varepsilon$. Choose $T^*=\max_{j=1, \cdots, n}\{T_j\}+\mu$, let $t>T^*$, then
	\begin{eqnarray*}
		e^{\lambda t}S_1(t)\leq e^{\lambda \tau(t)}\sum_{j=1}^{n}q_{ij}(t)e^{\lambda (t-\tau(t))}x_{j}(t-\tau(t))
		\leq\varepsilon e^{\lambda\mu}\sum_{j=1}^{n}\sup_{ t\geq\vartheta}|q_{ij}(t)|, 
	\end{eqnarray*} 
	which leads to $e^{\lambda t}S_1(t)\rightarrow0$ as $t\rightarrow\infty$. 
	
	By condition (iii), $e^{-\int_{0}^{t}\left[v_i(s)-\lambda\right]ds}\rightarrow0$ as $t\rightarrow\infty$, then we have
	\begin{eqnarray*}
		e^{\lambda t}S_2(t)&\leq &\left[\varphi_i(0)-\sum_{j=1}^{n}q_{ij}(0)x_{j}(-\tau(0))\right]e^{-\int_{0}^{t}\left[v_i(s)-\lambda\right]ds}, 
	\end{eqnarray*}
	which implies $e^{\lambda t}S_2(t)\rightarrow0$ as $t\rightarrow\infty$. 
	
	Then for any $\varepsilon>0$, there exists a nonimpulsive point $T_i>0$ such that $t>T_i$ implies $|e^{\lambda t}x_i(t)|<\varepsilon$, then
	\begin{eqnarray*}
		e^{\lambda t}S_3(t)&\leq&e^{\lambda t}\sum_{0\leq t_k\leq T_i}p_i(t_k-t_{k-1}) e^{-\int_{t_k}^{t}v_i(s)ds}\sum_{j=1}^{n}|x_j(t_k)|+e^{\lambda t}\sum_{T_i<t_k\leq t}p_i(t_k-t_{k-1}) e^{-\int_{t_k}^{t}v_i(s)ds}\sum_{j=1}^{n}|x_j(t_k)|\\
		&\leq&e^{-\int_{0}^{t}\left[v_i(s)-\lambda\right]ds}\sum_{0\leq t_k\leq T_i}p_i(t_k-t_{k-1}) e^{-\int_{t_k}^{0}v_i(s)ds}\sum_{j=1}^{n}|x_j(t_k)|+np_i\varepsilon\int_{T_i}^{t}e^{-\int_{s}^{t}v_i(u)du}ds\\
		&\leq&e^{-\int_{0}^{t}\left[v_i(s)-\lambda\right]ds}\sum_{0\leq t_k\leq T_i}p_i(t_k-t_{k-1}) e^{-\int_{t_k}^{0}v_i(s)ds}\sum_{j=1}^{n}|x_j(t_k)|+np_i\varepsilon\frac{1}{\eta_i}-np_i\varepsilon\frac{1}{\eta_i}e^{-\eta_it+\eta_iT_i}, 
	\end{eqnarray*}
	which implies $e^{\lambda t}S_3(t)\rightarrow0$ as $t\rightarrow\infty$. 
	
	Let $T^*=\max_{j=1, \cdots, n}\{T_j\}+\mu$, and $t>T^*$, we have
	\begin{eqnarray*}
		e^{\lambda t}S_4(t)
&\leq&e^{-\int_{0}^{t}(v_i(u)-\lambda)du}\int_{0}^{T^*}e^{-\int_{s}^{0}(v_i(u))du}\left[\sum_{j=1}^{n}\overline{c}_{ij}(s)x_{j}(s)+\sum_{j=1}^{n}a_{ij}(s)f_{j}(x_{j}(s))+\sum_{j=1}^{n}b_{ij}(s)g_{j}(x_{j}(s-\delta(s)))\right. \\
		&&+\sum_{j=1}^{n}w_{ij}(s)\int_{s-r(s)}^{s}h_{j}(x_{j}(u))\, duds
		\left. -\sum_{j=1}^{n}v_i(s)q_{ij}(s)x_j(s-\tau(s))\right]ds\\
&&+\frac{\varepsilon}{\eta_i-\lambda} \Bigg[\sum_{j=1}^{n}\sup_{ s\geq\vartheta}|\overline{c}_{ij}(s)|
	 +\sum_{j=1}^{n}\sup_{ s\geq\vartheta}|a_{ij}(s)\alpha_j|+e^{\lambda\mu}\sum_{j=1}^{n}\sup_{ s\geq\vartheta}|b_{ij}(s)\beta_j|+e^{\lambda\mu}\sum_{j=1}^{n}\sup_{ s\geq\vartheta}|w_{ij}(s)\mu\gamma_j|\\
&&+e^{\lambda\mu}\sum_{j=1}^{n}\sup_{ s\geq\vartheta}|v_i(s)q_{ij}(s)|\Bigg](1-e^{(\eta_i-\lambda)(T^*-t)}), 
	\end{eqnarray*}
	where $\xi(s)$ satisfies $s-r(s)\leq s-\xi(s)\leq s$, which results in $e^{\lambda t}S_4(t)\rightarrow0$ as $t\rightarrow\infty$. 
	
	Therefore, deduce $e^{\lambda t}\pi(x_i)(t)\rightarrow0$ as $t\rightarrow\infty$ for $i=1, \cdots, n$. We conclude that $\pi(x_i)(t)\subseteq\mathcal{H}(i=1, \cdots, n)$ which means $\pi(\mathcal{H})\subseteq\mathcal{H}$. 
	
	\textbf{Step 3}. Similar to the proof in Section 3, we see that $\pi$ is a contraction
	mapping, thus there exists a unique fixed point ${x}^{*}(\cdot)$ of
	$\pi$ in $\mathcal{H}$, which means $e^{\lambda t}\|{x}^{*}(\cdot)\|\rightarrow0$ as $t\rightarrow\infty$. 
	This completes the proof. 
\end{proof}
\begin{remark}
Zhang et al.{\rm{\cite{ZYT-1}}} and Chen et al.{\rm{\cite{GOS3}}} have investigated exponential stability of a special case of {\rm{(\ref{main equation})}} by using fixed point theory. The system studied in {\rm{\cite{ZYT-1, GOS3}}} has no neutral term, and all the coefficients are constants. Our results in Theorem {\rm{4.1}} improve and extend the result in {\rm{\cite{ZYT-1, GOS3}}}. 
\end{remark}

Consider the case when there are no impulsive effects, we obtain the following corollary. 
\begin{corollary}\label{cor2}
	Consider the nonlinear neutral delayed system {\rm{(\ref{main equation2})}}. Suppose that the assumptions {\rm{ (A1)-(A3)}} hold and the following conditions are satisfied:
	\begin{enumerate}[(i)]
		\item the delay $\delta(t)$, $\tau(t)$ and $r(t)$ are bounded by a positive constant $\mu$;
		\item there exist constants $\eta_i >0$ such that $v_i(t)>\eta _i$, $t\in \mathbb{R}^{+}$ for $i=1, 2, \cdots, n$;
		\item and such that 
		\begin{eqnarray*}
			\sum\limits_{i=1}^{n}\Bigg[\max_{j=1, \cdots, n}\sup_{t\geq\vartheta}|q_{ij}(t)|+\Bigg(\max_{j=1, \cdots, n}\sup_{t\geq\vartheta}|\overline{c}_{ij}(t)|+\max_{j=1, \cdots, n}\sup_{t\geq\vartheta}|a_{ij}(t)\alpha_j|+\max_{j=1, \cdots, n}\sup_{t\geq\vartheta}|b_{ij}(t)\beta_j|\\
			+\max_{j=1, \cdots, n}\sup_{t\geq\vartheta}| w_{ij}(t)\mu\gamma_j|+\max_{j=1, \cdots, n}\sup_{t\geq\vartheta}|q_{ij}(t)v_i(t)|\Bigg)\times\sup\limits_{t\geq\vartheta}\int_{0}^{t}e^{-\int_{s}^{t}v_i(u)du}ds\Bigg]<1. 
		\end{eqnarray*}
	\end{enumerate}
	Then the trivial equilibrium $x=0$ of system {\rm{(\ref{main equation2})}} is exponentially stable. 
\end{corollary}
\begin{remark}
Several exponential stability results {\rm{\cite{LLSZ, TZW, TZ}}}were provided for the special case of the system {\rm{(\ref{main equation2})}}, 
by constructing an appropriate Lyapunov functional and employing linear matrix inequality
(LMI) method. However, the delays in those results should satisfy the following condition: 

{\rm{(H)}} the discrete delay $\tau(t)$ is differentiable function and $r(t)$ in the distributed delay is nonnegative and bounded, that is, there exist constants $\tau_{M}$, $\zeta$, $r_{M}$ such that 

$$0\leq \tau(t)\leq\tau_{M}, \quad \tau'(t)\leq \zeta, \quad r(t)\leq r_{M}. $$

From our results, we provide other assumptions. The delays in our results are required to be bounded. Furthermore, Crollary {\rm{\ref{cor2}}} is an extension and improvement of the results in Chen et al.{\rm{\cite{GOS1}}} and Lai and Zhang{\rm{\cite{LZ}}}. 
\end{remark}

\section{Examples}
\begin{example}
	Consider the following two-dimensional impulsive neutral differential equations
	\begin{eqnarray}\label{example1}
		&&\left\{
		\begin{array}{ll}
			d\left[x(t)-Q(t)x(t-\tau(t))\right]=\left[Cx(t)+B(t)g(x(t-\delta(t)))+W\int_{t-r(t)}^{t}h(x(s))\, ds\right]\, dt, \quad t\geq 0, \quad t\neq t_{k}, \\
			&\\
			\Delta x(t_{k})=x(t_{k})-x(t^{-}_{k}), \quad t=t_{k}, \quad k=1, 2, 3, \cdots, 
		\end{array}
		\right. 
	\end{eqnarray}
	where 
	$$Q(t)=\begin{bmatrix}
		0. 1sin(t)&0\\
		0&0. 1cos(t)
	\end{bmatrix}, \quad C=\begin{bmatrix}
		-16&2. 5 \\
		1. 5&-16
	\end{bmatrix}, \quad B(t)=\begin{bmatrix}
		\frac{0. 3}{1+e^{-t}}&0 \\
		0&\frac{0. 4}{1+e^{-t}}
	\end{bmatrix}, \quad W=\begin{bmatrix}
		0. 2&0 \\
		0&0. 5
	\end{bmatrix}. $$
	with the initial conditions $x_1(t)=cos(t)$ , $x_2(t)=sin(t)$ on $-1\leq t\leq 0$ , where $\tau(t), \delta(t), r(t)=0. 2$ , $g_i(x)=\frac{|x+1|-|x-1|}{2}$ , $h_i(x)=sin(x)$ , $I_{ik}(t_k, (Fx)_i(tk))=arctan(0. 4x_i(t_k))$ , $t_k=t_{k-1}+0. 5k$ , $i=1, 2$ and $k=1, 2, \cdots$. 
\end{example}

We select $v_i(t)=16$ , it is clear that $\beta_i=\gamma_i=1$ , $p_{ik}=0. 4$ , $p_i=0. 8$ , $\eta_i=16$ for $i=1, 2$ and $k=1, 2, \cdots$. 

We check the condition (iv) in Theorem 3.1, 
\begin{eqnarray*}
			&&\sum\limits_{i=1}^{n}\Bigg[\max_{j=1, \cdots, n}\sup_{t\geq\vartheta}|q_{ij}(t)|+\Bigg(\max_{j=1, \cdots, n}\sup_{t\geq\vartheta}|\overline{c}_{ij}(t)|+\max_{j=1, \cdots, n}\sup_{t\geq\vartheta}|b_{ij}(t)\beta_j|+\max_{j=1, \cdots, n}\sup_{t\geq\vartheta}| w_{ij}(t)\mu\gamma_j|\\
			&&\quad+\max_{j=1, \cdots, n}\sup\limits_{\substack{j=1, \cdots, n\\t\geq\vartheta}}|q_{ij}(t)v_i(t)|+|p_i|\Bigg)\times\sup\limits_{t\geq\vartheta}\int_{0}^{t}e^{-\int_{s}^{t}v_i(u)du}ds\Bigg]=0. 7925<1. 
\end{eqnarray*}

 Hence, by using Theorem 3.1, we obtain that the trivial solution of (\ref{example1}) is asymptotically stable. Similarly, by using Theorem 4.1, the trivial solution of (\ref{example1}) is exponentially stable.

\begin{example}
	Consider the following two-dimensional impulsive neutral differential equations
	\begin{eqnarray} \label{example2}
		&&\left\{
		\begin{array}{ll}
			d\left[x(t)-Q(t)x(t-\tau(t))\right]=\left[Cx(t)+A(t)f(x(t))+B(t)g(x(t-\delta(t)))\right]\, dt, \quad t\geq 0, \quad t\neq t_{k}, \\
			&\\
			\Delta x(t_{k})=x(t_{k})-x(t^{-}_{k}), \quad t=t_{k}, \quad k=1, 2, 3, \cdots, 
		\end{array}
		\right. 
	\end{eqnarray}
	where 
	$$Q(t)=\begin{bmatrix}
		\frac{1}{8}sin^3(t)&0\\
		0&0.2sin(t)
	\end{bmatrix}, C=\begin{bmatrix}
		-18&-\frac{\sqrt{3}}{3} \\
		\frac{\sqrt{3}}{3}&-20
	\end{bmatrix}, A(t)=\begin{bmatrix}
		\frac{0. 01}{1+t}&0 \\
		0&0. 01e^{-t}
	\end{bmatrix}, B(t)=\begin{bmatrix}		
		\frac{999}{1000}cos(t)sin(2t)&0 \\
		0&\frac{999}{1000}cos^2(t)
	\end{bmatrix}. $$
	with the initial conditions $x_1(t)=0. 575t-0. 5$ , $x_2(t)=0. 7cos(t)$ on $-1\leq t\leq 0$ , where $\tau(t), \delta(t), r(t)=0. 2|sin(t)|$ , $f_i(x)=0. 2tanh(2x)$ , $g_i(x)=0. 6x$ , $I_{ik}(t_k, (Fx)_i(t_k))=arctan(0. 4x_i(t_k))$ , $t_k=t_{k-1}+0. 5k$ , $i=1, 2$ and $k=1, 2, \cdots$. 
	
\end{example}

Let $v_i(t)=20$, we have $\alpha=0. 4$, $\beta=0. 6$, $p_{ik}=0. 4$, $p_i=0. 8$, $\eta=20$ for $i=1, 2$ and $k=1, 2, \cdots$. 

Consequently, 
\begin{eqnarray*}
		&&\sum\limits_{i=1}^{n}\Bigg[\max_{j=1, \cdots, n}\sup_{t\geq\vartheta}|q_{ij}(t)|+\Bigg(\max_{j=1, \cdots, n}\sup_{t\geq\vartheta}|\overline{c}_{ij}(t)|+\max_{j=1, \cdots, n}\sup_{t\geq\vartheta}|b_{ij}(t)\beta_j|+\max_{j=1, \cdots, n}\sup_{t\geq\vartheta}| w_{ij}(t)\mu\gamma_j|\\
	&&\quad+\max_{j=1, \cdots, n}\sup\limits_{\substack{j=1, \cdots, n\\t\geq\vartheta}}|q_{ij}(t)v_i(t)|+|p_i|\Bigg)\times\sup\limits_{t\geq\vartheta}\int_{0}^{t}e^{-\int_{s}^{t}v_i(u)du}ds\Bigg]=0. 8832<1. 
\end{eqnarray*}

Then the condition (iv) in Theorem 3.1 holds, we conclude that the trivial solution of this two-dimensional impulsive neutral is asymptotically stable. Moreover, the trivial solution of (\ref{example2}) is exponentially stable. 

\begin{figure}[h]
	\centering
	\begin{minipage}{0.49\linewidth}
		\centering
		\includegraphics[scale=0.63]{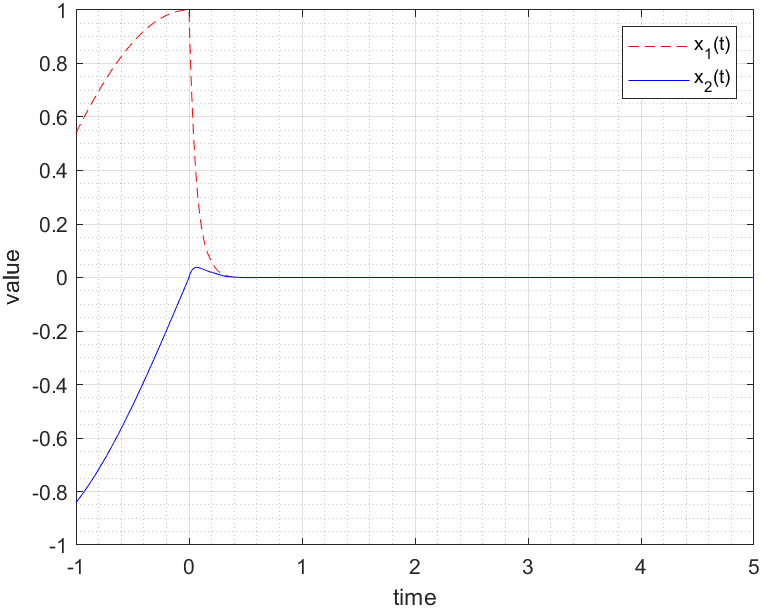}
		\caption{The solution of Example 5.1}
	\end{minipage}
	\begin{minipage}{0.49\linewidth}
		\centering
		\includegraphics[scale=0.63]{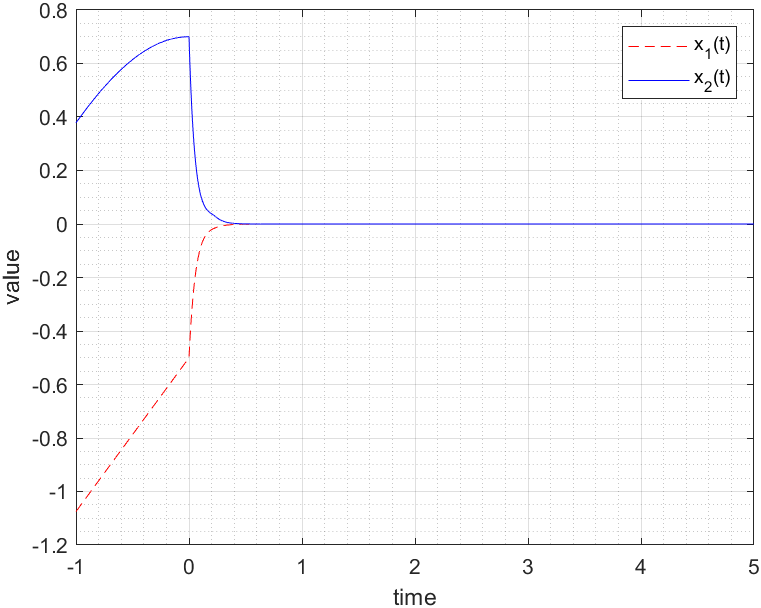}
		\caption{The solution of Example 5.2}
	\end{minipage}
\end{figure}

\end{document}